\documentclass[12pt,a4paper]{article}
\usepackage[utf8]{inputenc}
\usepackage[T1]{fontenc}
\usepackage{amsmath}
\usepackage{amssymb}
\usepackage{graphicx}
\usepackage{tikz}
\usepackage{comment}

\newtheorem{theorem}{Theorem}[section]

\newtheorem{example}[theorem]{Example}

\newtheorem{lemma}[theorem]{Lemma}

\newtheorem{corollary}[theorem]{Corollary}

\newtheorem{definition}[theorem]{Definition}
\newenvironment{proof}[1][Proof]{\noindent\textbf{#1.} }{\ \rule{0.5em}{0.5em}}

\begin{document}
	\begin{center}
		\textbf{\large{$\mathcal{I}^\mathcal{K}$-limit points, $\mathcal{I}^\mathcal{K}$-cluster points and $\mathcal{I}^\mathcal{K}$-Fr\'{e}chet compactness }}
\end{center}

	\begin{center}
	Manoranjan Singha$^*$, Sima Roy$^{**}$
	\end{center}

	\footnotetext{	
		

		$^*$Department of Mathematics, University of North Bengal, Darjeeling-734013, India, E-mail: manoranjan.math@nbu.ac.in
	
	$^{**}$Department of Mathematics, Raja Rammohun Roy Mahavidyalaya, Hooghly-712406, India, E-mail: rs$\_$sima@nbu.ac.in
	}

		\textbf{Abstract:} In 2011, the theory of $\mathcal{I}^\mathcal{K}$-convergence gets birth as an extension of the concept of $\mathcal{I}^*$-convergence of sequences of real numbers. $\mathcal{I}^\mathcal{K}$-limit points and $\mathcal{I}^\mathcal{K}$-cluster points of functions are introduced and studied to some extent, where $\mathcal{I}$ and $\mathcal{K}$ are ideals on a non-empty set $S$. In a first countable space set of $\mathcal{I}^\mathcal{K}$-cluster points is coincide with the closure of all sets in the filter base $\mathcal{B}_f(\mathcal{I}^\mathcal{K})$ for some function $f:S\to X$. Fr\'{e}chet compactness is studied in light of ideals $\mathcal{I}$ and $\mathcal{K}$ of subsets of $S$ and showed that in $\mathcal{I}$-sequential $T_2$ space Fr\'{e}chet compactness and $\mathcal{I}$-Fr\'{e}chet compactness are equivalent. A class of ideals have been identified for which $\mathcal{I}^\mathcal{K}$-Fr\'{e}chet compactness coincides with $\mathcal{I}$-Fr\'{e}chet compactness in first countable spaces.
		
		\begin{comment}	
		showed that Fr\'{e}chet compactness and $\mathcal{I}$-Fr\'{e}chet compactness are equivalent in first countable $T_1$ space provided $\mathcal{I}$ be an ideal on $\omega$. In a first countable space  $\mathcal{I}$-Fr\'{e}chet compactness and $\mathcal{I}^\mathcal{K}$-Fr\'{e}chet compactness are equivalent provided ideals satisfied shrinking condition $A(\mathcal{I}, \mathcal{K})$.
		\end{comment}

\section{Reading the past}
	
	Theory of ideal convergence of real numbers was introduced by  P. Kostyrko, T. \v{S}al\'{a}t, W. Wilczy\'{n}ski \cite{Kostyrkoa}, which is an extension of statistical convergence by H. Fast \cite{Fast}( also Schoenberg \cite{Schoenberg}).  On this convergence many results have been done by different mathematicians like  B. K. Lahiri \cite{Lahiria}, J. Cincura \cite{Cincura}, T. \v{S}al\'{a}t \cite{Salatb}, S. Pehlivan \cite{Pehlivan}, M. Balcerzak \cite{Balcerzak}, K. Demirci \cite{Demirci} etc. Also the authors of \cite{Kostyrkoa} introduced $\mathcal{I}^*$-convergence, influenced by a result of T. Sal\'at \cite{T.Salat} and J. A. Fridy \cite{Fridy} about statistical convergent sequences (a sequence of real numbers is statistically convergent to $x$ if and only if there exists a set $M\subset \mathbb{N}$ with $\delta{(M)}=1$, asymptotic density or natural density of M \cite{Niven} such that the corresponding subsequence converges to $x$), and showed that  $\mathcal{I}^*$-convergence always implies $\mathcal{I}$-convergence but $\mathcal{I}$-convergence does not imply $\mathcal{I}^*$-convergence. As a result, a class of ideals (property AP) is defined in \cite{Kostyrkoa} and it is seen that the notions $\mathcal{I}$-convergence and $\mathcal{I}^*$-convergence are equivalent if and only if the ideal $\mathcal{I}$-satisfies the property AP. 
	
	An ideal $\mathcal{I}$ on an arbitrary set $S$ is a family $\mathcal{I}\subset 2^S$ that is closed under finite unions and taking subsets \cite{Kuratowski}.  An ideal $\mathcal{I}$ is called trivial if $\mathcal{I}= \{\emptyset\}$ or $S$ in $\mathcal{I}.$ A non-trivial ideal $\mathcal{I}\subset 2^S$ is called admissible if it contains all the singleton sets \cite{Kuratowski}. The ideal Fin is the class of all finite subsets of $\mathbb{N}$. Various examples of non-trivial admissible ideals are given in \cite{Kostyrkoa}. The notion of $\mathcal{I}$-limit point and $\mathcal{I}$-cluster point in a metric space $(X,d)$ were defined by P. Kostyrko, T. \v{S}al\'{a}t, W. Wilczy\'{n}ski \cite{Kostyrkoa}. B. K. Lahiri, P. Das \cite{Lahirib} generalized these concepts in  topological spaces and characterized $\mathcal{I}(C_x)$, the collection of all $\mathcal{I}$-cluster points of a sequence $x=(x_n)$ in a topological space $X$, as closed subsets of $X$ (see Theorem 10, \cite{Lahirib}). Also for any ideal $\mathcal{I}$ on $\mathbb{N}$,  $\mathcal{I}(L_x)$ the collection of all $\mathcal{I}$-limit points of $x$ is a subset of $\mathcal{I}(C_x)$ (see Theorem 9, \cite{Lahirib}). 
	
	In 2011, the $\mathcal{I}^\mathcal{K}$-convergence of function in a topological space was introduced by Ma\v{c}aj and Sleziak \cite{Macaj} as a generalization of $\mathcal{I}^*$-convergence of function.
	
	Let $S$ be a non-empty set and let $\mathcal{I}$ be an ideal on $S$. A function $f$ from $S$ to a topological space $X$ is said to be $\mathcal{I}^*$-convergent to some point $a\in X$ if there exists a set $K\in\mathcal{F}(\mathcal{I})$ such that the function $g:S\to X$ defined by
	$$g(s)=\left\{
	\begin{array}{ll}
		f(s) & \hbox{,~if~$s\in K$} \\
		a & \hbox{,~otherwise}
	\end{array}
	\right.$$ 
	is Fin-convergent to $a$. Whenever $S=\mathbb{N}$, $\mathcal{I}^*$-convergence of functions coincide with $\mathcal{I}^*$-convergence of sequences as a special case. $\mathcal{I}^\mathcal{K}$-convergence is defined by replacing Fin by an arbitrary ideal $\mathcal{K}$ on $S$. If $\mathcal{I}$ is a nontrivial admissible ideal on S, $\mathcal{I}/_M =\{A \cap M;A \in \mathcal{I}\}$
	is an ideal on M called trace of $\mathcal{I}$ on $M$ \cite{Macaj}. $\mathcal{I}/_M$ is nontrivial if $M\notin \mathcal{I}$.

\section{$\mathcal{I}^\mathcal{K}$-Limit Points and $\mathcal{I}^\mathcal{K}$-Cluster points}
Let's recall the notions of $\mathcal{I}^\mathcal{K}$-limit point and  $\mathcal{I}^\mathcal{K}$-cluster point for a function in a topological space $X$ and  results from \cite{Sharmah}.
\begin{definition}
	Let $f: S\rightarrow X$ be a function and $\mathcal{I}$, $\mathcal{K}$ be two ideals on $S$.
	\begin{enumerate}
		\item  $x\in X$ is called an $\mathcal{I}^\mathcal{K}$-limit point of $f$ if there exists a set $M\in \mathcal{F}(\mathcal{I})$ such that for the function $g:S\rightarrow X$ defined by $g/_M= f/_M$ and $g[S\setminus M]=\{x\}$ has a $\mathcal{K}$-limit point $x$. 
		\item  $x\in X$ is called an $\mathcal{I}^\mathcal{K}$-cluster point of $f$ if there exists a set $M\in \mathcal{F}(\mathcal{I})$ such that for the function $g:S\rightarrow X$ defined by $g/_M= f/_M$ and $g[S\setminus M]=\{x\}$ has a $\mathcal{K}$-cluster point $x$, i.e., for any open set $U$ containing $x$,  $\{s\in S: g(s)\in U\}\notin \mathcal{K}$.
	\end{enumerate}
\end{definition}

The collection of all $\mathcal{I}^\mathcal{K}$-limit points and $\mathcal{I}^\mathcal{K}$-cluster points of a function $f$ in a topological space $X$ is denoted by $L_f(\mathcal{I}^\mathcal{K})$ and $C_f(\mathcal{I}^\mathcal{K})$ respectively. For admissible ideals $\mathcal{I}$ and $\mathcal{K}$, $L_f(\mathcal{I}^\mathcal{K})\subset C_f(\mathcal{I}^\mathcal{K})$ \cite{Sharmah}.

\begin{theorem} $(i)$ $C_f(\mathcal{I}^\mathcal{K})$ is closed for any function $f$ from S to a topological space $X$ and any ideals $\mathcal{I}$, $\mathcal{K}$ on S.\\
	$(ii)$ Suppose $X$ is completely separable and $\mathcal{I}$, $\mathcal{K}$ be ideals on $S$. If there exists a pairwise disjoint sequence of sets $\{T_p\}$ such that $T_p\subset S$, $T_p\notin
	\mathcal{K}$ for all $p\in \omega$ and $\mathcal{I}\subset\mathcal{K}$ then for any non-empty closed set $C\subset X$, there is a function $f: S\to X$ such that $C=C_f(\mathcal{I}^\mathcal{K})$.

\end{theorem}
\begin{proof} The proof of $(i)$ is omitted as being similar to the proof of Theorem $4.6(i)$ in \cite{Sharmah}. For $(ii)$, since $X$ is completely separable, $C$ is separable and so let $A=\{c_1,c_2,...\}\subset C$ be a countable set with $\bar{A}=C$. Consider a function $f: S\to X$ such that $$f(s)=\left\{
	\begin{array}{ll}
		c_i & \hbox{,~if~$s\in T_i;$} \\
		b & \hbox{,~otherwise}
	\end{array}
	\right.$$
	b is an element of $C$. Let $p\in C_f(\mathcal{I}^\mathcal{K})$ (taking $p\neq b$ and $c_i$ for any i, otherwise if $p=b$ or $c_i$ for some i, then $p\in C$). Let $U$ be any open set containing $p$. Then there exists a set $M\in\mathcal{F}(\mathcal{I})$ such that function $g: S\to X$ given by $g/_M= f/_M$ and $g[S\setminus M]=\{p\}$, has $\mathcal{K}$-cluster point $p$. Therefore $\{s\in S: g(s)\in U\}\notin \mathcal{K}$. Since $\mathcal{I}\subset\mathcal{K}$, $\{s\in S: g(s)\in U\}\notin \mathcal{I}$ and so either $b\in U$ or $c_i\in U$, for some i. Therefore $U\cap C\neq \phi$. Thus $p$ is a limit point of $C$ and so $p\in C$. Hence $C_f(\mathcal{I}^\mathcal{K})\subset C$. 
	
	Conversely, let $p\in C$ and let $U$ be any open set containing $p$. Then there is $c_i\in A$ such that $c_i\in U$. Thus $T_i\subset \{s\in S: f(s)\in U\}$. Therefore there exists a set $M\in\mathcal{F}(\mathcal{I})$ such that function $g: S\to X$ given by $g/_M= f/_M$ and $g[S\setminus M]=\{p\}$ and $\{s\in S: g(s)\in U\}\supset \{s\in S: f(s)\in U\}$. Since $\{s\in S: f(s)\in U\}\notin \mathcal{K}$,  $\{s\in S: g(s)\in U\}\notin \mathcal{K}$. This implies that $p\in C_f(\mathcal{I}^\mathcal{K})$.
\end{proof}

\begin{theorem} \label{fin} Let $f: S\rightarrow X$ and $g: S\rightarrow X$ be two functions and $\mathcal{I}$, $\mathcal{K}$, $\mathcal{K}_1$,$\mathcal{K}_2$ be ideals on $S$. Then\\
	$(i)$ $C_f(\mathcal{I}^\mathcal{K}) \subset C_f(Fin)$ and  $L_f(\mathcal{I}^\mathcal{K}) \subset L_f(Fin)$.\\
	$(ii)$ $C_f(\mathcal{I}^{\mathcal{K}_2}) \subset C_f(\mathcal{I}^{\mathcal{K}_1})$ and $L_f(\mathcal{I}^{\mathcal{K}_2}) \subset L_f(\mathcal{I}^{\mathcal{K}_1})$ with $\mathcal{K}_1\subset \mathcal{K}_2$.\\
		$(iii)$ If $\{s\in S : f(s)\neq g(s)\}\in\mathcal{K}$, then $C_f(\mathcal{I}^\mathcal{K})=$ $C_g(\mathcal{I}^\mathcal{K})$ and   $L_f(\mathcal{I}^\mathcal{K})=$ $L_g(\mathcal{I}^\mathcal{K})$.
	
\end{theorem}

\begin{proof} $(i)$ and $(ii)$ follow from the definitions.	For $(iii)$, 
let $p\in C_f(\mathcal{I}^\mathcal{K})$. Then there exists a set $M\in \mathcal{F}(\mathcal{I})$ such that the function $h_1: S\to X$ given by $h_1(s)=f(s)$ if $s\in M$ and $p$ otherwise, has $\mathcal{K}$-cluster point $p$. So for any open set $U$ containing $p$, $\{s\in S : h_1(s)\in U\}\notin \mathcal{K}$. Consider a function $h_2: S\to X$  given by $h_2(s)=g(s)$, if $s\in M$ and $p$ otherwise. But if $\{s\in S : h_2(s)\in U\}\in \mathcal{K}$, then $\{s\in S : h_1(s)\in U\}\subset \{s\in S : h_2(s)\in U\}\cup \{s\in S : h_1(s)\neq h_2(s)\}$ so by given condition and our assumption  $\{s\in S : h_1(s)\in U\}\in\mathcal{K}$, which is a contradiction. Therefore $\{s\in S : h_2(s)\in U\}\notin \mathcal{K}$ and so $p \in C_g(\mathcal{I}^\mathcal{K})$. Similarly, $C_g(\mathcal{I}^\mathcal{K})\subset C_f(\mathcal{I}^\mathcal{K})$ and consequently $C_f(\mathcal{I}^\mathcal{K})=$ $C_g(\mathcal{I}^\mathcal{K})$.
The proof of $\mathcal{I}^\mathcal{K}$-limit points is same as previous one.
\end{proof}

If $(X, \tau)$ is a topological space , then
$\{X\setminus C\subset X ; C =  {\underset {f\in C^A} \bigcup} C_f(\mathcal{I}^\mathcal{K})\}$ forms a topology on $X$, denoted by $\tau(\mathcal{I}^\mathcal{K})$, cf. \cite{Leonetti}.

\begin{theorem}\label{topol}
	For any topological space $(X, \tau)$, $\tau\subset \tau(\mathcal{I}^\mathcal{K})$. In addition, if $X$ is first countable, $\tau = \tau(\mathcal{I}^\mathcal{K})$.
	
\end{theorem}

\begin{proof}
	Let $G$ be $\tau$-closed set. For each element $\alpha \in G$, consider the constant function $f: A \to G$ at $\alpha$. Then $G\subset {\underset {f\in G^A} \bigcup} C_f(\mathcal{I}^\mathcal{K}) \subset  {\underset {f\in G^A} \bigcup} C_f(Fin) =G$ (as in Theorem \ref{fin} (i)).\\
	Consider a $\tau({\mathcal{I}^\mathcal{K}})$-closed set $G$ and $\alpha\in G$. There is a set $M\in \mathcal{F}(\mathcal{I}/_A)$ such that the function $g: A\to G$ given by 
	$$g(s)=\left\{
	\begin{array}{ll}
		f(s) & \hbox{,~if~$s\in M;$} \\
		\alpha & \hbox{,~if~$s\in A\setminus M$}
	\end{array}
	\right.$$
	has $\mathcal{K}$-cluster point $\alpha$. Let $\{U_p\}$ be a decreasing local base at $\alpha$. Then $E_i=\{s\in A: g(s)\in U_i\}\notin \mathcal{K}/_A$, for each $i\in\mathbb{N}$. Take $s_1\in E_1$ and for each $i\in\mathbb{N}$, $s_{i+1}\in$ $E_{i+1}\setminus \{s_1,s_2,...s_{i}\}$. Suppose $C=\{s_1,s_2,...\}$, then $\alpha$ is a limit point of $\{g(s): s\in C\}$ that is, $\alpha$ is a limit point of $G$. Hence $G$ is $\tau$-closed set.
\end{proof}

\begin{lemma}\label{Unifcont}
	Let $A$ be compact subset of a topological space $X$ and let  $f: S \to X$ be a function. Then $\{s\in S : f(s) \in A\}\notin \mathcal{K}$ implies that $A\cap C_f(\mathcal{I}^\mathcal{K})\neq \phi$.
\end{lemma}

\begin{proof}
	If possible let, $A\cap C_f(\mathcal{I}^\mathcal{K})= \phi$. Let $a\in A$. Then there exists a set $M\in \mathcal{F}(\mathcal{I})$ such that the function $g: S \to X$ given by $g(s)=f(s)$ if $s\in M$ and $a$ otherwise, has no $\mathcal{K}$-cluster point. So for each $a\in A$ there exists an open set $U_a$ containing $a$ such that $\{s\in S : g(s)\in U_a\}\in\mathcal{K}$. Since $\{s\in S : f(s)\in U_a\}\subset \{s\in S : g(s)\in U_a\}$, so $L_a=\{s\in S : f(s)\in U_a\}\in\mathcal{K}$. Now consider the collection $\{U_a: a\in A\}$ which forms an open cover of the compact set $A$ and so it has a finite subcover $\{U_{a_1}, U_{a_2},...U_{a_l}\}$ (say). Then $\{s\in S: f(s)\in A\}\subset L_{a_1}\cup L_{a_2}\cup...\cup L_{a_l}\in\mathcal{K}$. This implies $\{s\in S: f(s)\in A\}\in\mathcal{K}$, which is a contradiction.
\end{proof}

For any topological space $X$, $f:S\to X$ be a function then $\mathcal{I}^\mathcal{K}$-filter generated by $f$ is defined by\\
 $\mathcal{G}_f(\mathcal{I}^\mathcal{K})=$ $\{Y\subset X:$ there is $A\in\mathcal{F}(\mathcal{I})$ such that $\{s\in S: g(s)\notin Y\} \in \mathcal{K} \}$,
 where $$g(s)=\left\{
\begin{array}{ll}
	f(s) & \hbox{,~if~$s\in A;$} \\
	y & \hbox{,~otherwise}
\end{array}
\right.$$, $y\in X$. Therefore  $\mathcal{G}_f(\mathcal{I}^\mathcal{K})$ forms a filter on $X$, cf. \rm(\cite{Bourbaki}, Definition 7, p.64). The corresponding filter base
$\mathcal{B}_f(\mathcal{I}^\mathcal{K})= \{ \{g(s): s\notin K\}: K\in\mathcal{K}$ and there is $A\in \mathcal {F}(\mathcal{I})$ with $ g/_A=f/_A $ and $g(S\setminus A)=\{y\}, y\in X\}$, cf. \rm(\cite{Bourbaki}, Definition 2, p.69).

\begin{theorem}
	For any topological space $X$, suppose $f:S\to X$ be a function  and $\mathcal{I}$, $\mathcal{K}$ ideals on $S$. Then 
	${\underset {B\in\mathcal{B}_f(\mathcal{I}^\mathcal{K})} \bigcap} \bar{B}\subset C_f(\mathcal{I}^\mathcal{K}) $. In addition if $X$ is first countable, 	${\underset {B\in\mathcal{B}_f(\mathcal{I}^\mathcal{K})} \bigcap} \bar{B}= C_f(\mathcal{I}^\mathcal{K})$.
	
\end{theorem}
\begin{proof}
	Let $p\in{\underset {B\in\mathcal{B}_f(\mathcal{I}^\mathcal{K})} \bigcap} \bar{B}$. Suppose $p$ is not an $\mathcal{I}^\mathcal{K}$-cluster point of $f$. Then for any set $L\in \mathcal{F}(\mathcal{I})$ such that function 
	 $$h(s)=\left\{
	\begin{array}{ll}
		f(s) & \hbox{,~if~$s\in L;$} \\
		p & \hbox{,~otherwise}
	\end{array}
	\right.$$
	 has no $\mathcal{K}$-cluster point. So there exists an open set $U$ containing $p$ such that $L_1=\{s\in S : h(s)\in U\}\in\mathcal{K}$. This implies that $\{h(s): s\notin L_1\}\in\mathcal{B}_f(\mathcal{I}^\mathcal{K})$. It follows $p\in{\underset {B\in\mathcal{B}_f(\mathcal{I}^\mathcal{K})} \bigcap} \bar{B}\subset  \overline {\{h(s): s\notin L_1\}} =\{h(s) : h(s)\notin U\}\subset X\setminus U$ that is $p\in X\setminus U$,  which leads a contradiction.
	Conversely suppose $p\in C_f(\mathcal{I}^\mathcal{K})$. Then there exists a set $M\in \mathcal{F}(\mathcal{I})$ such that function $g: S\to X$ given by 
	 $$g(s)=\left\{
	\begin{array}{ll}
		f(s) & \hbox{,~if~$s\in M;$} \\
		p & \hbox{,~otherwise}
	\end{array}
	\right.$$
	has a $\mathcal{K}$-cluster point $p$. Let $\{W_i\}$ be a decreasing local base at $p$. Then for each $i$, $A_i=\{s\in S: g(s)\in W_i\}\notin\mathcal{K}$. Take $K\in\mathcal{K}$ and $B_i=A_i\setminus K\notin\mathcal{K}$, for each $i$. Put $i_1\in B_1$ and $i_{n+1}\in B_{n+1}\setminus \{i_1,i_2,...i_{n}\}$, for all $n$. Suppose $L=\{i_1,i_2,...\}$, then clearly $i_n\notin K$. Then $p$ is a limit point of the set $\{g(s): s\notin K\}$ that is $p\in \overline {\{g(s): s\notin K\}}$. Hence $p\in {\underset {k\in \mathcal{K}} \bigcap} \overline {\{g(s): s\notin K\}}$.
\end{proof}

Consider a Hausdorff uniform space $(Y, \mathbb{U})$. Suppose $K(Y)$ and $CL(Y)$ are the collection of all nonempty compact and closed subsets of $X$ respectively. Recall that a sequence $(y_n)$ in $(Y, \mathbb{U})$ is said to be  bounded if $\{y_n: n\in\omega\}\subset V[a]=\{b\in Y : (a,b)\in V\}$ for some $a\in Y$ and $V\in\mathbb{U}$ and $(Y, \mathbb{U})$ is said to be boundedly compact if every closed bounded subset in $Y$ is compact. 
Now denote by $bs(Y)$ the set of all bounded sequences in $(Y, \mathbb{U})$ and by $cs(Y)$ the set of all sequences $y=(y_n)$ in $Y$ with $C_y(\mathcal{I}^\mathcal{K})\neq \phi$. According to Lemma \ref{Unifcont}, for every sequences $y\in bs(Y)$,  $C_y(\mathcal{I}^\mathcal{K})\neq \phi$ i.e., $bs(Y)\subset cs(Y)$. If $(Y,\mathbb{U})$ is boundedly compact, then  $C_y(\mathcal{I}^\mathcal{K})$ is compact (it is closed and bounded). Hence the assignment $y\mapsto C_y(\mathcal{I}^\mathcal{K})$ defines a mapping $\Gamma_{\mathcal{I}^\mathcal{K}} $ of the set $bs(Y)$ to the set $K(Y)$ of all nonempty compact subsets of  $(Y, \mathbb{U})$. As in \cite{Das1}, endow $cs(Y)$ with a uniformity $\tilde{\mathbb{U}}$ defined by

$\tilde{\mathbb{U}}=\{\tilde{V}=((y_n),(z_n)):$ for all $n$, $(y_n\in V[z_n]$ and $z_n\in V[y_n]), V\in \mathbb{U}\}$ 

and the set $K(Y)$ with the Hausdorff-Bourbaki uniformity $\mathbb{U}_H$ inherited from the space $CL(Y)$ defined by 

$\mathbb{U}_H=\{\underline{V}=(B, C)\in CL(Y)\times CL(Y) : V\in \mathbb{U}, B\subset V[C]$ and $C\subset V[B]\}$.

Next we recall the definition of Vietoris topology on the set $\mathcal{A}$ of all non empty closed subsets of a Hausdorff topological space $Y$. For any subset $B$ of $Y$ put 

$B^-=\{C\in \mathcal{A} : C\cap B\neq \phi \}$ and $B^+= \{C\in \mathcal{A} : C\subset B\}$.

Then the upper Vietories topology on $\mathcal{A}$ denoted by $\tau_V^+$, is the topology with subbase $S_B^-=\{B^+ : B$ is open in $Y\}$ and the lower Vietories topology on $\mathcal{A}$ denoted by $\tau_V^+$, is the topology with subbase $S_B^-=\{B^- : B$ is open in $Y\}$. The Vietories topology on $\mathcal{A}$ denoted by $\tau_V$, is the topology with subbase $S_B^-\cup S_B^+$.

\begin{theorem}
	Suppose $(X, \mathbb{U})$ is boundedly compact uniform space. Then the mapping $\Gamma_{\mathcal{I}^\mathcal{K}} : (bs(X),\tilde{\mathbb{U}})\to (K(X), \mathbb{U}_H)$ is uniformly continuous.
\end{theorem}

\begin{proof}
	Suppose $U\in \mathbb{U}$. Then there exists $V\in \mathbb{U}$ such that $V^3\subset U$. Let $((x_n), (y_n))\in\tilde{V}$ and $a \in C_x(\mathcal{I}^\mathcal{K})$, where $x=(x_n)$. Then there exists a set $M\in \mathcal{F}(\mathcal{I})$ such that sequence $(z_n)$ given by $z_n=x_n$ if $n\in M$ and a otherwise, has $\mathcal{K}$-cluster point $a$. Therefore $K= \{n\in\omega : z_n\in V[a]\}\notin \mathcal{K}$. Consider a sequence $(w_n)$ given by  $w_n=y_n$ if $n\in M$ and a otherwise. Since for each n, $(x_n, y_n)\in V$ which implies $(z_n, w_n) \in V$. So for each $n\in K$, $(z_n,a)\in V$ and $(z_n, w_n) \in V$ and then $(w_n, a)\in V^2$ i.e., $w_n\in V^2[a]$. Hence $K\subset \{n\in\omega : w_n\in V^2[a]\}$. As $K\notin\mathcal{K}$, $\{n\in\omega : w_n\in V^2[a]\}\notin\mathcal{K}$. Since $\overline{V^2[a]}$ is compact, by Lemma \ref{Unifcont},  $\overline{V^2[a]}\cap C_y(\mathcal{I}^\mathcal{K})\neq \phi$ (where $y=(y_n))$, which implies $a\in V^3[C_y(\mathcal{I}^\mathcal{K})]$. Therefore $C_x(\mathcal{I}^\mathcal{K})\subset V^3[C_y(\mathcal{I}^\mathcal{K})]\subset U[C_y(\mathcal{I}^\mathcal{K})]$. Similarly $C_y(\mathcal{I}^\mathcal{K})\subset U[C_y(\mathcal{I}^\mathcal{K})]$ and it follows that $(C_x(\mathcal{I}^\mathcal{K}), C_y(\mathcal{I}^\mathcal{K}))\in \underline{U}\in \mathbb{U}_H$.
\end{proof}

\begin{theorem}
	Suppose $(X,\mathbb{U})$ is locally compact uniform space. Then the mapping $\Gamma_{\mathcal{I}^\mathcal{K}} : (cs(X),\tilde{\mathbb{U}})\to (CL(X), \tau_V ^-)$ is continuous.
\end{theorem}

\begin{proof}
	Suppose $O^-$ be a basic open set in $(CL(X), \tau_V ^-)$, where $O$ is open in $X$. Let $x=(x_n)\in \Gamma_{\mathcal{I}^\mathcal{K}}^{-1}(O^-)$. Then $C_x(\mathcal{I}^\mathcal{K})\cap O\neq \phi$. Let $a\in C_x(\mathcal{I}^\mathcal{K})\cap O$. Since  $(X,\mathbb{U})$ is locally compact, there exists $U\in\mathbb{U}$ such that $\overline{U[a]}$ is compact and $a\in \overline{U[a]}\subset O$. So for $U\in\mathbb{U}$, there exists $V\in \mathbb{U}$ such that $V^2\subset U$. As  $a\in C_x(\mathcal{I}^\mathcal{K})$, there exists a set $M\in\mathcal{F}(\mathcal{I})$ such that sequence $(y_n)$ given by $y_n=x_n$ if $n\in M$ and a otherwise, has $\mathcal{K}$-cluster point $a$. So $B=\{n\in \omega : y_n\in V[a]\}\notin \mathcal{K}$. Let  $z=(z_n)\in \tilde{V}[(x_n)]$, then for each $n$, $(x_n, z_n)\in V$. Consider a sequence  $(w_n)$ given by  $w_n=z_n$ if $n\in M$ and a otherwise. Again for each $n\in B$, $y_n\in V[a]$ and $w_n\in V[y_n]$ which implies $w_n\in V^2[a]\subset U[a]$. Therefore $B\subset \{n\in \omega : w_n\in U[a]\}\subset  \{n\in \omega : w_n\in \overline{U[a]}\}$. As $B\notin\mathcal{K}$, $\{n\in \omega : w_n\in \overline{U[a]}\}\notin \mathcal{K}$. Since $\overline{U[a]}$ is compact, by Lemma \ref{Unifcont}, $\overline{U[a]}\cap C_z(\mathcal{I}^\mathcal{K})\neq \phi$ and so $C_z(\mathcal{I}^\mathcal{K})\cap O\neq \phi$. Hence $z\in \Gamma_{\mathcal{I}^\mathcal{K}}^{-1}(O^-)$ and consequently $\tilde{V}[(x_n)]\subset  \Gamma_{\mathcal{I}^\mathcal{K}}^{-1}(O^-)$. Hence  $\Gamma_{\mathcal{I}^\mathcal{K}}^{-1}(O^-)$ is open in $(cs(X),\tilde{\mathbb{U}})$.
\end{proof}
\section{$\mathcal{I}$-Fr\'{e}chet Compactness and $\mathcal{I}^\mathcal{K}$-Fr\'{e}chet Compactness}

Let's define $\mathcal{I}$-nonthin function as a common generalization of $\mathcal{I}$-nonthin sequence stated in \cite{Singha}.
\begin{definition}
	 Suppose $X$ is a topological space and $A\subset S$. A function $f: A\to X$ is said to be $\mathcal{I}$-thin, where $\mathcal{I}$ is an ideal on $S$ if $A\in \mathcal{I}$, otherwise it is called $\mathcal{I}$-nonthin.
\end{definition}

\begin{definition}
	For a subset $C$ of a topological space $X$ and $x\in X$, the $\mathcal{I}$-closure of $C$ is denoted by  $\overline{C}^{\mathcal{I}}=$ $\{x\in X:$ there exists an $\mathcal{I}$-nonthin function $ f: A \to C$ that  $\mathcal{I}/_A$-converges to $x\}$ and the $\mathcal{I}^\mathcal{K}$-closure of $C$ is denoted  by $\overline{C}^{\mathcal{I}^\mathcal{K}}=$ $\{x\in X:$ there exists an $\mathcal{I}$-nonthin function $ f: A \to C$ that  $(\mathcal{I}/_A)^\mathcal{K}$-converges to $x\}$.
\end{definition}
	
\begin{theorem} \label{topolI}
	
	Suppose $\mathcal{I}$ , $\mathcal{K}$ are ideals on $S$ such that $\mathcal{K}\subset \mathcal{I}$. For any subset $C$ of a topological space $X$, $C \subset \overline{C}^{\mathcal{K}} \subset  \overline{C}^{\mathcal{I}^\mathcal{K}} \subset \bar{C}$. Furthermore, if $X$ is first countable, $\mathbb{N}\subset S$ and $\mathbb{N}\notin \mathcal{I}$, $\bar{C} = \overline{C}^{\mathcal{K}} =\overline{C}^{\mathcal{I}^\mathcal{K}} $
\end{theorem}
\begin{proof}
	Consider $\alpha \in \overline{C}^{\mathcal{I}^\mathcal{K}}$, there exists an $\mathcal{I}$-nonthin function $f: A \to C$ that  $(\mathcal{I}/_A)^\mathcal{K}$-converges to $\alpha$. Therefore there is $M\in \mathcal{F}(\mathcal{I}/_A)$ such that the function $g: A \to C$ given by 
	$$g(s)=\left\{
	\begin{array}{ll}
		f(s) & \hbox{,~if~$s\in M;$} \\
		\alpha & \hbox{,~if~$s\in A\setminus M$}
	\end{array}
	\right.$$ 
	is $\mathcal{K}$-convergent to $\alpha$. So for any open set $U$ containing $\alpha$, $\{s\in A: g(s)\in U\}\in \mathcal{F}(\mathcal{K}/_A)$. Since $\mathcal{K}\subset \mathcal{I}$, the set $\{s\in A: g(s)\in U\}\in \mathcal{F}(\mathcal{I}/_A)$ and so $\{s\in A: f(s)\in U\}\in \mathcal{F}(\mathcal{I}/_A)$. Then there exists $p\in A$ such that $p\in \{s\in A: f(s)\in U\}$. Thus $f(p)\in C\cap U$ and hence  $f(p)\in \bar{C}$.
	Now suppose $\alpha \in \bar{C}$. Then there exists a function $f: \mathbb{N} \to C$ such that  $f: \mathbb{N} \to C$ is convergent to $\alpha$. Since $\mathcal{I}$, $\mathcal{K}$ admissible ideal on $S$,  $f: \mathbb{N} \to C$ is $\mathcal{I}$-convergent as well as $\mathcal{K}$-convergent to $\alpha$. Thus $\alpha \in \overline{C}^{\mathcal{I}^\mathcal{K}}$.
\end{proof}
\begin{definition}
	A subset $C$ of a topological space $X$ is called $\mathcal{I}$-closed if $\overline{C}^{\mathcal{I}}=C$ and $\mathcal{I}^\mathcal{K}$-closed if $\overline{C}^{\mathcal{I}^\mathcal{K}}=C$. 
\end{definition}

	  Theorem \ref{topolI} follows that closed subsets with $\mathcal{K} \subset\mathcal{I}$ are $\mathcal{I}^\mathcal{K}$-closed. If $(X,\tau)$ is a topological space, then $\tau_{\mathcal{I}^\mathcal{K}}$ $=\{G\subset X: X\setminus G$ is $\mathcal{I}^\mathcal{K}$-closed$\}$ forms a topology on $X$ having $\tau \subset \tau_{\mathcal{I}^\mathcal{K}}$. Therefore combining Theorem \ref{topol} and Theorem \ref{topolI}  $\tau_{\mathcal{I}^\mathcal{K}}$ = $\tau({\mathcal{I}^\mathcal{K}})$ in a first countable space.

\begin{definition}
	An $\mathcal{I}$-nonthin function $f: A\to X$ is $\mathcal{I}$-eventually constant at $\alpha$ if $\{s\in A: f(s)\neq \alpha\}\in\mathcal{I}/_A$.
\end{definition}

\begin{definition}
	A point $\alpha$ in a topological space $X$ is said to be an
	 $\mathcal{I}_{ev}$-limit point of $Y\subset X$ if there exists an $\mathcal{I}$-nonthin non $\mathcal{I}$-eventually constant function $f: A\rightarrow Y\setminus\{\alpha\}$ that $\mathcal{I}/_A$-converges to $\alpha$ and $\mathcal{I}^\mathcal{K}_{ev}$-limit point of $Y\subset X$ if there exists an $\mathcal{I}$-nonthin non $\mathcal{I}$-eventually constant function $f: A\rightarrow Y\setminus\{\alpha\}$ that $(\mathcal{I}/_A)^\mathcal{K}$-converges to $\alpha$.
\end{definition}

\begin{definition}
	A topological space $X$ is said to be $\mathcal{I}$-Fr\'{e}chet compact if every infinite subset of $X$ has an $\mathcal{I}_{ev}$-limit point and is called $\mathcal{I}^\mathcal{K}$-Fr\'{e}chet compact if every infinite subset of $X$ has an $\mathcal{I}^\mathcal{K}_{ev}$-limit point.
\end{definition}

 \begin{theorem}\label{charac}
 	A topological space $X$ is not $\mathcal{I}^\mathcal{K}$-Fr\'{e}chet compact if and only if there exists an infinite $\mathcal{I}^\mathcal{K}$-closed discrete subspace. 
 \end{theorem}
\begin{proof}
	Let $G$ be an infinite $\mathcal{I}^\mathcal{K}$-closed discrete subspace. Then $\overline{G}^{\mathcal{I}^\mathcal{K}}=G$. If $\alpha$ is an  $\mathcal{I}^\mathcal{K}_{ev}$-limit point of $G$, there exists an $\mathcal{I}$-nonthin non $\mathcal{I}$-eventually constant function $f:A\to G\setminus \{\alpha\}$ that $(\mathcal{I}/_A)^\mathcal{K}$-converges to $\alpha$. There is $B\in \mathcal{F}(\mathcal{I}/_A)$ such that the function $g: B \to G\setminus \{\alpha\}$ given by $g(s)=f(s), s\in B$ is $\mathcal{K}$-convergent to $\alpha$. Since $G$ is a discrete subspace, there is an open set $U$ containing $\alpha$, $\{s\in B: g(s)\in U\}$ is an empty set, which is a contradiction. Therefore $X$ is not  $\mathcal{I}^\mathcal{K}$-Fr\'{e}chet compact space. Now let $A\subset X$ be an infinite $\mathcal{I}^\mathcal{K}$-closed discrete subspace. If possible let, $X$ is  $\mathcal{I}^\mathcal{K}$-Fr\'{e}chet compact, then $A$ has an $\mathcal{I}^\mathcal{K}_{ev}$-limit point say $l$. But since $A$ is $\mathcal{I}^\mathcal{K}$-closed, $l\in A$. Also $A$ is discrete subspace, which contradicts the fact that $l$ is an $\mathcal{I}^\mathcal{K}_{ev}$-limit point. 
\end{proof}
 
 \begin{corollary}
 	A topological space $X$ is $\mathcal{I}^\mathcal{K}$-Fr\'{e}chet compact if and only if every $\mathcal{I}^\mathcal{K}$-closed discrete subspaces are finite.
 \end{corollary}

 \begin{corollary}
 	In a topological space $\mathcal{I}^\mathcal{K}$-Fr\'{e}chet compactness with $\mathcal{K}\subset\mathcal{I}$ implies Fr\'{e}chet compactness.
 \end{corollary}

The following theorem is just a translation of the proof of the Theorem \ref{charac} in terms of $\mathcal{I}$-Fr\'{e}chet compact.

\begin{theorem}
A topological space $X$ is not $\mathcal{I}$-Fr\'{e}chet compact if and only if there exists an infinite $\mathcal{I}$-closed discrete subspace. 
\end{theorem}

\begin{theorem}
	${\mathcal{I}^\mathcal{K}}$- Fr\'{e}chet compactness implies $\mathcal{I}$- Fr\'{e}chet compactness, provided $\mathcal{K}\subset\mathcal{I}$.
\end{theorem}
\begin{proof}
	Suppose $X$ is ${\mathcal{I}^\mathcal{K}}$- Fr\'{e}chet compact space and $Y$ is infinite subset of $X$. Then $Y$ has an ${\mathcal{I}^\mathcal{K}}_{ev}$-limit point say $\alpha$, that is there is an $\mathcal{I}$-nonthin non $\mathcal{I}$-eventually constant function $f: A\to Y\setminus \{\alpha\}$ that $\mathcal{K}$-converges to $\alpha$. Since $\mathcal{K}\subset \mathcal{I}$, $\alpha$ is a $\mathcal{I}_{ev}$-limit point of $Y$.
\end{proof}

\begin{example}
	Consider the space $X=\{\frac{1}{n}; n\in \mathbb{N}\}\cup \{0\}$ where discrete topology on $\{\frac{1}{n}; n\in \mathbb{N}\}$ and cofinite topology containing $0$.  Let $\mathcal{I}=\{A\subset \mathbb{N} :A \cap\Delta_i$ is finite for all but finitely many i$\}$, where $\mathbb{N}={\overset{\infty}{\underset{j=1}\cup}} \Delta_j$ be a decomposition of $\mathbb{N}$ such that each $\Delta_j$ is infinite and $\Delta_i\cap\Delta_j=\phi$ for $i\neq j$ and $\mathcal{K}=Fin$. Then for any infinite subset $A$ of $X$, $0$ is an $\mathcal{I}_{ev}$-limit point of $A$. Thus $X$ is $\mathcal{I}$-Fr\'{e}chet compact. But there is no $\mathcal{I}$-nonthin non $\mathcal{I}$-eventual constant sequence that is $\mathcal{K}$-convergent, so $X$ is not $\mathcal{I}^\mathcal{K}$-Fr\'{e}chet compact.
\end{example}

\begin{definition}
	Let $\mathcal{I}$, $\mathcal{K}$ be ideals on $S$. $\mathcal{I}$ is said to satisfy shrinking condition(A) with respect to $\mathcal{K}$ or shrinking condition A$(\mathcal{I},\mathcal{K})$ holds, if for any sequence $\{A_i\}$ of sets none in $\mathcal{I}$, there exists a sequence $\{B_i\}$ of sets in $\mathcal{K}$ such that $B_i\subset$ $A_i$ and ${\overset{\infty}{\underset{i=1}\cup}} B_i\notin \mathcal{I}$.
\end{definition}
The following example is an witness of such ideal.
\begin{example}
	\textup{Let $\mathbb{N}={\overset{\infty}{\underset{j=1}\cup}} \Delta_j$ be a decomposition of $\mathbb{N}$ such that each $\Delta_j$ is infinite and $\Delta_i\cap\Delta_j=\phi$ for $i\neq j$. Consider $\mathcal{I}=$ $\{A\subset \mathbb{N} :A \cap\Delta_i$ is finite for all but finitely many i$\}$ and  $\mathcal{K}$ denote the class of all subset $A$ of $\mathbb{N}$ which intersect at most finite number of $\Delta_j$s \cite{Kostyrkoa}. Then $\mathcal{I}$ satisfies shrinking condition(A) with respect to $\mathcal{K}$.
	}	
\end{example}

\begin{theorem}
	Suppose $X$ is first countable space and $\mathcal{I}$, $\mathcal{K}$ are admissible ideals on $S$. If shrinking condition A$(\mathcal{I},\mathcal{K})$ holds, then $\mathcal{I}$-Fr\'{e}chet compactness implies $\mathcal{I}^\mathcal{K}$- Fr\'{e}chet compactness.
\end{theorem}

\begin{proof}
	Suppose $X$ is first countable $\mathcal{I}$- Fr\'{e}chet compact space. Consider an infinite subset $A\subset X$ which has an $\mathcal{I}_{ev}$-limit  point say $\alpha \in X$. There is an $\mathcal{I}$-nonthin function $f: B\to A\setminus \{\alpha\}$ that $\mathcal{I}/_B$-converges to $\alpha$. Let $\{U_n\}$ be countable base for $X$ at the point $\alpha$ such that $U_{n+1}\subset U_n$, for all $n\in \mathbb{N}$. Therefore for all $m\in \mathbb{N}$, 
	$A_m=\{s\in B : f/_B(s)\in U_m\}\notin \mathcal{I}$. If shrinking condition $A(\mathcal{I},\mathcal{K})$ holds, there exists a sequence of sets $\{B_i\}$ in $\mathcal{K}$ such that $B_i\subset A_i$ and $D=\underset{i\in \mathbb{N}}\cup B_i\notin\mathcal{I}$. Let $U$ be any open set containing $\alpha$, then there exists $U_p\in\{U_n\}$ such that $U_n\subset U$, for all $n\geq p$. So $\{s\in D : f(s)\notin U\}\subset B_1\cup B_2\cup ...\cup B_p\in \mathcal{K}$. Therefore the restriction $f: D\to X\setminus \{\alpha\}$ is $\mathcal{K}$-convergent to $\alpha$. 
\end{proof}

\begin{example}
	Consider the space $X=\mathbb{N}\times \{0,1\}$, where discrete topology on $\mathbb{N}$ and indiscrete topology on $\{0,1\}$. Clearly $X$ is Fr\'{e}chet compact. Now if we take an infinite set $A=\{(1,0), (2,0), (3,0),....\}$, then it has no $\mathcal{I}_{ev}$-limit point.
\end{example}

\begin{theorem}
	For any ideal $\mathcal{I}$ on $\omega$,
	first countable Fr\'{e}chet compact $T_1$ space is  $\mathcal{I}$-Fr\'{e}chet compact.
\end{theorem}

\begin{proof}
Let A be any infinite subset of a first countable limit point compact $T_1$ space. Then A has a limit point say $\alpha$. Let $\{V_n\}$ be a countable base at $\alpha$ such that $V_{n+1}\subset V_n$, for all n. Therefore there exists a sequence $(x_n)$ such that $x_n\in A\cap (V_n\setminus \{\alpha\})$ and $(x_n)$ converges to $\alpha$. This implies that $(x_n)$ is non $\mathcal{I}$-eventually constant and $\mathcal{I}$-converges to $\alpha$. So $\alpha$ is an $\mathcal{I}_{ev}$-limit point of $A$.
\end{proof}
\begin{corollary}\label{equiv}
	For any first countable $T_1$ space, if $\mathcal{I}$ be ideal on $\omega$, Fr\'{e}chet compactness and $\mathcal{I}$-Fr\'{e}chet compactness are equivalent.
\end{corollary}
\begin{theorem}
	 $\mathcal{I}^\mathcal{K}$-closed $(\mathcal{I}$-closed$)$ subset of  $\mathcal{I}^\mathcal{K}$-Fr\'{e}chet compact space $(\mathcal{I}$-Fr\'{e}chet compact space$)$ is  $\mathcal{I}^\mathcal{K}$-Fr\'{e}chet compact $($resp. $\mathcal{I}$-Fr\'{e}chet compact $)$.
\end{theorem}
\begin{proof}
Consider an $\mathcal{I}^\mathcal{K}$-closed subset Y of  $\mathcal{I}^\mathcal{K}$-Fr\'{e}chet compact space X. Let $A\subset Y$ be infinite. Then $A$ has an  $\mathcal{I}^\mathcal{K}_{ev}$-limit point say, $\alpha$. There is an $\mathcal{I}$-nonthin function $f: B\to A\setminus\{\alpha\}$ that $\mathcal{K}$-converges to $\alpha$. Since $A\subset Y$, $\alpha\in \overline{Y}^{\mathcal{I}^\mathcal{K}}=Y$. Therefore $A$ has an $\mathcal{I}^\mathcal{K}_{ev}$-limit point in $Y$ and so $Y$ is $\mathcal{I}^\mathcal{K}$-Fr\'{e}chet compact.
\end{proof}

\begin{corollary}
	Closed subspace of  $\mathcal{I}^\mathcal{K}$-Fr\'{e}chet $(\mathcal{I}$-Fr\'{e}chet compact space$)$ compact space is  $\mathcal{I}^\mathcal{K}$-Fr\'{e}chet compact  $($resp. $\mathcal{I}$-Fr\'{e}chet compact $)$.
\end{corollary}

The following theorem possessed by finite derived set property or FDS-property, which was introduced in \cite{Wilson} to study some Fr\'{e}chet of the lattice of $T_1$-topologies on a set. 

\begin{definition}\cite{Wilson}
	A topological space $X$ has the finite derived set property or FDS-property if every infinite subset of X contains an infinite subset with only finitely many limit points in X. 
\end{definition}

\begin{theorem}
	For any topological space with the FDS-property, Fr\'{e}chet compactness implies $\mathcal{I}$-Fr\'{e}chet compactness. 
\end{theorem}
\begin{proof}
Suppose $(X,\tau)$ be a Fr\'{e}chet compact space with the FDS-property and $Y$ be an infinite subset of $X$. Since $X$ has the FDS-property, there exists an infinite subset $A$ of $Y$ with finite set of limit points say $\{\alpha_1, \alpha_2,...,\alpha_n\}$. Let $f: M_0\to A$ be an $\mathcal{I}$-nonthin function of distinct elements. If  $f: M_0\to A$ is $\mathcal{I}/_{M_0}$-convergent to $\alpha_1$, then the proof is done. If not then there exists an open set $U_1$ containing $\alpha_1$ such that $M_1=\{s\in M_0 : f(s)\notin U_1\}\notin \mathcal{I}$. Let $A_1=\{f(s); s\in M_1\}$. Then $A_1$ is an infinite set and  $f: M_1\to A_1$ be an $\mathcal{I}$-nonthin function of distinct elements.  If  $f: M_1\to A_1$ is $\mathcal{I}/_{M_1}$-convergent to $\alpha_2$, then the proof is done. If not then there exists an open set $U_2$ containing $\alpha_2$ such that $M_2=\{s\in M_1 : f(s)\notin U_2\}\notin \mathcal{I}$. Therefore $\{f(s); s\in M_2\}$ is an infinite set say $A_2$. Proceeding in this way we get for some $k\leq n$, the $\mathcal{I}$-nonthin function $f: M_{k-1} \to A_{k-1}$ is $\mathcal{I}/_{M_{k-1}}$-converges to $\alpha_k$. Otherwise the infinite set $A\setminus{\overset{n}{\underset{i=1}\cup}} A_i$ has no limit point, which contradicts our assumption. Hence every infinite subset of $X$ has an $\mathcal{I}_{ev}$-limit point.
\end{proof}

As in Corollary \ref{equiv} Fr\'{e}chet compactness and $\mathcal{I}$-Fr\'{e}chet compactness are equivalent in first countable $T_1$ topological space, where $\mathcal{I}$ is an ideal on $\omega$. In the following theorem it is established on the large class of $\mathcal{I}$-sequential space.

\begin{definition}
	A topological space is called $\mathcal{I}$-sequential if every $\mathcal{I}$-closed set is closed.
\end{definition}

\begin{theorem}
	Every $\mathcal{I}$-sequential Hausdorff Fr\'{e}chet compact space is $\mathcal{I}$-Fr\'{e}chet compact.
\end{theorem}

\begin{proof}
	Suppose $X$ is $\mathcal{I}$-sequential Hausdorff Fr\'{e}chet compact space. Let $A$ be an infinite subset of $X$ having no $\mathcal{I}_{ev}$-limit point. Thus $A$ is $\mathcal{I}$-closed and so closed. Since closed subset of a countable compact space is countable compact, $A$ is countable compact and Hausdorff. Therefore $A$ is first countable in the relative topology and so $\mathcal{I}$-Fr\'{e}chet compact, which contradicts the assumption.
\end{proof}

In \cite{Singha} $\mathcal{I}$-compactness was introduced and showed that even in metric spaces $\mathcal{I}$-compactness and compactness are different. A topological space $X$ is said to be $\mathcal{I}$-compact if any $\mathcal{I}$-nonthin function $f:A\to X$ has an $\mathcal{I}$-nonthin restriction $f/_B: B\to X$ that $\mathcal{I}/_B$-converges to some point in $X$. For any nontrivial ideal on $S$, every $\mathcal{I}$-compact space is $\mathcal{I}$-Fr\'{e}chet compact. As in Note 4 \cite{Singha}, if $\mathcal{I}_m$ is the dual maximal ideal to the free ultrafilter, then every Fr\'{e}chet compact space is $\mathcal{I}$-compact.
 Now in the following, figure 1 is described the relation among all such compactness.

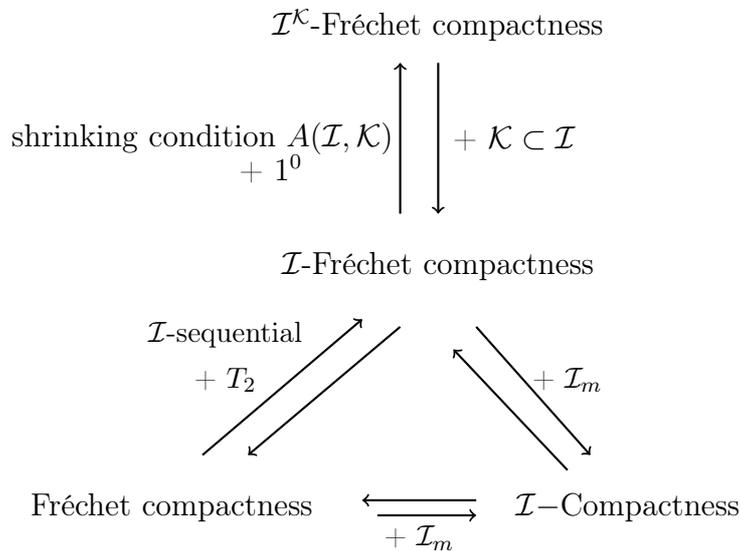
\begin{figure}[ht]
	\begin{center}
		\begin{tikzpicture}
			\node[] at (-2,0.1) {Fr\'{e}chet compactness};
			\draw[thick,->] (0.7,0) -- (2,0);
			\draw[thick,<-] (0.5,0.2) -- (2,0.2);
			\node[] at (1.25,-.3) {\small + $\mathcal{I}_m$};
			\node[] at (-1.3,2.4) {\small{$\mathcal{I}$-sequential} };
			\node[] at (-1.3,1.8) {\small{+ $T_2$} };
				\node[] at (3.2,1.8) {\small{+ $\mathcal{I}_m$} };
			\node[] at (4,0.1) {$\mathcal{I}-$Compactness};
			\draw[thick,->] (2,2.5) -- (3.5,0.8);
			\draw [thick,<-] (1.7,2.2) -- (3.2,0.6);
			\node[] at (2.5,5) { + $\mathcal{K}\subset\mathcal{I}$};
			\node[] at (-1.6,5) { shrinking condition $A(\mathcal{I},\mathcal{K})$};
			\node[] at (-.7,4.6) {+ $1^0$};
			\draw[thick,->] (1,2.5) -- (-1,0.8);
			\draw[thick,<-] (0.5,2.6) -- (-1.6,0.8);
			\node[] at (1.5,3.3) {$\mathcal{I}$-Fr\'{e}chet compactness};
			\draw[thick,<-] (1,6) -- (1,4);
			\draw[thick,->] (1.5,6) -- (1.5,4);
			\node[] at (1.5,6.5) {$\mathcal{I}^\mathcal{K}$-Fr\'{e}chet compactness};
			\end{tikzpicture}
	      	\caption{Relation among $\mathcal{I}^\mathcal{K}$-Fr\'{e}chet compactness, $\mathcal{I}$-Fr\'{e}chet compactness, $\mathcal{I}$-compactness and Fr\'{e}chet compactness}
	\label{fig}
\end{center}
\end{figure}

\end{document}